\documentclass{llncs}

\newcommand{\Sym}{\ensuremath \mathrm{Sym}}

\newcommand{\lcm}{\ensuremath \mathop \mathrm{lcm} \nolimits}

\begin{document}

\title{Multiple Factorizations of
Bivariate Linear Partial Differential Operators}
\titlerunning{Multiple Factorizations of
Bivariate Linear Partial Differential Operators}

\author{Ekaterina Shemyakova}
\authorrunning{E.Shemyakova}
\institute{Research Institute for Symbolic Computation (RISC),\\
J.Kepler University,\\
Altenbergerstr. 69, A-4040 Linz, Austria\\
\email{kath@risc.uni-linz.ac.at},\\ WWW home page:
\texttt{http://www.risc.uni-linz.ac.at} }

\maketitle

\begin{abstract}
We study the case when a bivariate Linear Partial Differential
Operator (LPDO) of orders three or four has several different
factorizations.

We prove that a third-order bivariate LPDO has a first-order left
and right factors such that their symbols are co-prime if and only
if the operator has a factorization into three factors, the left one
of which is exactly the initial left factor and the right one is
exactly the initial right factor. We show that the condition that
the symbols of the initial left and right factors are co-prime is
essential, and that the analogous statement ``as it is'' is not true
for LPDOs of order four.

Then we consider completely reducible LPDOs, which are defined as an
intersection of principal ideals. Such operators may also be
required to have several different factorizations. Considering all
possible cases, we ruled out some of them from the consideration due
to the first result of the paper. The explicit formulae for the
sufficient conditions for the complete reducibility of an LPDO were
found also.
\end{abstract}

\section{Introduction}

The factorization of Linear Partial Differential Operators (LPDOs)
is an essential part of recent algorithms for the exact solution for
Linear Partial Differential Equations (LPDEs). Examples of such
algorithms include numerous generalizations and modifications of the
18th-century Laplace Transformations
Method~\cite{2ndorderparab,ts:genLaplace05,ts:steklov_etc:00,anderson_juras97,anderson_kamran97,athorne1995,Zh-St,St08},
the Loewy decomposition method~\cite{gs,gs:genloewy05,gs:08}, and
others.

The problem of constructing a general factorization algorithm for an
LPDO is still an open problem, though several important
contributions have been made over the last decades (see for
example~\cite{gs,LiSchTs:03,LiSchTs02,ts:enumerLODO:96,ts:genLaplace05,obstacle2}).
The main difficulty in the case of LPDOs is non-uniqueness of
factorization: (irreducible) factors and the number of factors are
not necessarily the same for two different factorizations of the
same operator. For example, for the famous Landau
operator~\cite{blumberg} $L$ we have $ L = (D_x + 1 +
\frac{1}{x+c(y)} ) \circ (D_x + 1 -
  \frac{1}{x+c(y)} ) \circ (D_x + x D_y) =(D_{xx} +x D_{xy} +D_x +(2+x)D_y )\circ (D_x + 1) \ .$
Note that the second order factor in the second factorization is
hyperbolic and is irreducible.

However, for some classes of LPDOs factorization is unique. For
example, there is~\cite{gs} no more than one factorization that
extends a factorization of the principal symbol of the operator into
co-prime factors (see Theorem~\ref{thm:gs}).

Some important methods of exact integration, for example, mentioned
above Loewy decomposition methods require LPDOs to have a number of
different factorizations of certain types. Also completely reducible
LPDOs introduced in~\cite{gs}, which becomes significant as the
solution space of a completely reducible LPDO coincides with the sum
of those of its irreducible right factors may require a number of
right factors. Thus, in Sec.~\ref{sec:main} we study the case when a
bivariate (not necessarily hyperbolic) LPDO has two different
factorizations. For operators of order three we have a really
interesting result (Theorems~\ref{thm:main_3} and
\ref{thm:improved1}). We showed that analogous statement for
operators of order four is not true.

For the proof of the theorems we use invariants' methods. Invariants
of LPDOs under the gauge transformations (see Sec.~\ref{sec:def})
are widely used for factorization problems since Laplace' times as
many properties appearing in connection with the factorization of an
LPDO are invariant under the gauge transformations, and, therefore,
can be expressed in terms of generating invariants, which were found
in~\cite{movingframes}. Factorization itself is invariant under the
gauge transformations: if for some LPDO $L$, $L= L_1 \circ L_2$,
then $L^g = L_1^g \circ L_2^g$. Expressions for necessary and
sufficient conditions for the existence of a factorization of a
given LPDO of a given factorization type were found
in~\cite{inv_cond_hyper_case} and ~\cite{inv_cond_non_hyper_case}.
We use these expressions in the proofs of the theorems of
Sec.~\ref{sec:completely_reducible}.

Theorems~\ref{thm:main_3} and \ref{thm:improved1} of
Section~\ref{sec:main} allow us to reduce consideration of cases in
Sec.~\ref{sec:completely_reducible}, where we show how the problem
of the complete reducibility of a hyperbolic bivariate LPDO can be
expressed in terms of invariants also.

\section{Definitions and Notations}
\label{sec:def}

Consider a field $K$ of characteristic zero with commuting
derivations $\partial_x, \partial_y$, and the ring of linear
differential operators $K[D]=K[D_x, D_y]$, where $D_x, D_y$
correspond to the derivations $\partial_x,
\partial_y$, respectively. In $K[D]$ the variables $D_x, D_y$ commute
with each other, but not with elements of $K$. For $a\in K$ we have
$D_i a = aD_i + \partial_i(a)$. Any operator $L \in K[D]$ has the
form $L = \sum_{i+j =0}^d a_{ij} D_x^i D_y^j$, where $a_{ij} \in K$.
The polynomial $\Sym(L) =  \sum_{i+j = d} a_{ij} X^i Y^j$ in formal
variables $X, Y$ is called the (principal) \emph{symbol} of $L$.

Below we assume that the field $K$ is differentially closed unless
stated otherwise, that is it contains solutions of (non-linear in
the generic case) differential equations with coefficients from $K$.

Let $K^*$ denote the set of invertible elements in $K$. For $L \in
K[D]$ and every $g \in K^*$ consider the gauge transformation $L
\rightarrow L^g = g^{-1} \circ L \circ g$. Then an algebraic
differential expression $I$ in the coefficients of $L$ is
\emph{invariant} under the gauge transformations (we consider only
these in the present paper) if it is unaltered by these
transformations. Trivial examples of invariants are the coefficients
of the symbol of an operator. A generating set of invariants is a
set using which all possible differential invariants can be
expressed.

Given a third-order bivariate LPDO $L$ and a factorization of its
symbol  $\Sym(L)$ into first-order factors. In some system of
coordinates the operator has one of the following normalized forms:
\begin{eqnarray} \label{L1}
L &=& (p(x,y) D_x + q(x,y) D_y) D_x D_y  + \sum_{i+j=0}^2
a_{ij}(x,y) D_x^iD_y^j \ ,
\\
\label{op:XXY} L &=& D^2_x D_y  + \sum_{i+j=0}^2 a_{ij}(x,y)
D^i_{x}D_{y}^j \ , \\
\label{op:XXX} L &=& D^3_x + \sum_{i+j=0}^2 a_{ij}(x,y)
D^i_{x}D_{y}^j \ ,
\end{eqnarray}
where $p=p(x,y) \neq 0$, $q=q(x,y) \neq 0$, $a_{ij}=a_{ij}(x,y)$.
The normalized form~(\ref{L1}) has symbol $S=pX+qY$. Without loss of
generality one can assume $p=1$. Each of the class of operators
admits gauge transformations $L \rightarrow g^{-1} \circ L \circ g$
for $g=g(x,y) \neq 0$.

\begin{remark} Recall that since for two LPDOs $L_1, L_2 \in K[D]$ we have
$\Sym(L_1 \circ L_2) = \Sym(L_1) \cdot \Sym(L_2)$, any factorization
of an LPDO extends some factorization of its symbol. In general, if
$L \in K[D]$ and $\Sym(L)=S_1 \cdot \dots \cdot S_k$, then we say
that the factorization
\[
L=F_1 \circ \dots  \circ F_k, \quad \quad \Sym(F_i0=S_i, \ \forall i
\in \{1, \dots, k\},
\]
is of the {\it factorization type} $(S_1)\dots(S_k)$.
\end{remark}
We reformulate the famous result of~\cite{gs} in the new notation:
\begin{theorem}~\cite{gs} \label{thm:gs} Let LPDO $L$ of arbitrary order and in arbitrary
number of independent variables have symbol $\Sym(L)=S_1 \dots S_k$,
where $S_i$-s are pairwise co-prime. Then there exists at most one
factorization of $L$ of the type $(S_1)\dots(S_k)$.
\end{theorem}

\section{Factorization via Invariants for Hyperbolic  Bivariate
Operators of Order Three}
\label{sec_a_full_system_of_inv}

Information from this section will be used in the proofs below in
the case, where $L$ is hyperbolic operator.

\begin{theorem} \cite{invariants_gen} \label{invariant_general}
The following $7$ invariants form a generating set of invariants for
operators of the form (\ref{L1}): $q$, $I_1
=2q^2a_{20}-qa_{11}+2a_{02}$, $I_2= -qa_{02y}+a_{02}q_y+q^2a_{20x}$,
$I_3=a_{10}+2q_ya_{20}+a_{20}^2q-a_{11y}+qa_{20y}-a_{11}a_{20}$,
$I_4=a_{01}q^2-3q_xa_{02}+a_{02}^2-a_{11x}q^2+a_{11}qq_x+qa_{02x}
-a_{02}a_{11}q$, $I_5
=a_{00}q+2a_{02}a_{20x}-a_{02}a_{10}-a_{01}a_{20}q
-\frac{1}{2}a_{11xy}q+qq_xa_{20y}-a_{11}qa_{20x}+qq_ya_{20x}+2q^2a_{20}a_{20x}+qq_{xy}a_{20}
+a_{20}a_{11}a_{02}$.
\end{theorem}

The set of values of these seven invariants uniquely defines an
equivalent class of operators of the form~(\ref{L1}). Also invariant
properties of such operators can be described in terms of the seven
invariants.
\begin{lemma} \label{fact_is_invariant}
The property of having a factorization
(or a factorization extending a certain factorization of the symbol)
is invariant.
\end{lemma}
\begin{proof} Let $L=F_1 \circ F_2 \circ \dots \circ F_k$, for some operators
$F_i \in K[D]$.
 For every $g \in K^*$ we have
$g^{-1} \circ L \circ g = \left( g^{-1} \circ F_1 \circ g \right)
\circ  \left( g^{-1} \circ F_2 \circ g \right) \circ \dots \circ
\left( g^{-1} \circ F_k \circ g \right)$.
\end{proof}

Looking through the formulaes of the next theorem, notice that some
conditions are the same for different types of factorizations. In
particular, one can pay attention to conditions $(A_1)-(D_1)$. Such
correlations will be used in the next section (Sec.~\ref{sec:main}).

\begin{theorem}~\cite{inv_cond_hyper_case} \label{thm:non-constant}
Given the values of the invariants $q, I_1,I_2,I_3,I_4,I_5$ (from
Theorem~\ref{invariant_general}) for an equivalence class of
operators of the form~(\ref{L1}). The LPDOs of the class have a
factorization of factorization type
\begin{description}

  \item[$(S)(XY)$]  if and only if
\begin{equation} \label{conds:(X+qY)(XY)}
\left.
\begin{array}{ll}
& I_3q^3-I_{1y}q^2+q_yI_1q-I_4+qI_{1x}-2q_xI_1-3qI2 = 0 , \\
& -q^2I_{4y}+1/2q^3I_{1xy}-qI_{4x}-3/2q^2q_xI_{1y}+q^3I_5+
q^2I_{1xx} \\
&-3/2I_1q^2q_{xy}-  2I_1qq_{xx}+5I_1qq_xq_y+
6I_1q_x^2+3I_4q_x\\
&+3I_4qq_y-qI_1I_{1x}+I_1I_4+2q_xI_1^2 -4I_{1x}qq_x-3/2I_{1x}q^2q_y
\\
&-2q^2I_{2x}-q^3I_{2y}+I_2qI_1+
 4I_2qq_x+2I_2q^2q_y = 0 \ ;
\end{array}
\right\}
\end{equation}

\item[$(S)(X)(Y)$] if and only if
$ \quad (\ref{conds:(X+qY)(XY)}) \quad \& \quad
-I_4+qI_{1x}-2q_xI_1-qI_2= 0$ \ ;

\item[$(S)(Y)(X)$] if and only if
$\quad (\ref{conds:(X+qY)(XY)}) \quad \& \quad (C_1)$ \ ;

\item[$(X)(SY)$] if and only if
\begin{equation} \label{conds:(X)(qSY)}
\left.
\begin{array}{ll}
& (D_1): \quad qq_{xx}-I_4-2q_x =0 \ , \\
& -3/2q_xqI_{1y}-q^3I_{3x}+I_5q^2+1/2q^2I_{1xy}-1/2qq_yI_{1x}+
\\
& q_xq^2I_3+2I_1q_xq_y-1/2I_1q_{xy}q -4q_xI_2+qI_{2x} = 0 \ .
\end{array}
\right\}
\end{equation}

\item[$(X)(S)(Y)$] if and only if
$ \quad (\ref{conds:(X)(qSY)}) \quad \& \quad (B_1)$;
\item[$(X)(Y)(S)$] if and only if
$ \quad (\ref{conds:(X)(qSY)}) \quad \& \quad (D_1)$;

\item[$(XY)(S)$] if and only if
\begin{equation} \label{conds:(XY)(qS)}
\left.
\begin{array}{ll}
& -qI_2+qq_xq_y+q_{yy}q^3-q^2q_{xy}+qq_{xx}+I_3q^3-I_4-2q_x^2
= 0 \ , \\
& q^3I_5+qI_{4x}+1/2q^3I_{1xy}-3/2q^2q_xI_{1y}+I_1I_4+q^2I_{2x}+
 2I_1qq_xq_y \\
&+2I_1q_x^2-5I_4q_x-1/2I_1q^2q_{xy}-I_1qq_{xx}+
 I_4qq_y-1/2I_{1x}q^2q_y -4I_2qq_x\\
& -10q_x^3-q^2q_{xxx}-q^4I_{3x}
 +I_3q^3q_x+2qq_x^2q_y-q^2q_yq_{xx}+8qq_xq_{xx}= 0 \ .
\end{array}
\right\}
\end{equation}

\item[$(YS)(X)$] if and only if
\begin{equation} \label{conds:(YqS)(X)}
\left.
\begin{array}{ll}
& (C_1): \quad -2q_xI_1+qI_{1x}-I_4-2qI_2 = 0 \ , \\
& -qI_{4y}+1/2q^2I_{1xy}+I_5q^2-I_2I_1-q^2I_{2y}+2q_yI_4+3I_1q_xq_y-
\\
& 3/2I_1q_{xy}q-1/2qq_yI_{1x}-3/2q_xqI_{1y} = 0 \ .
\end{array}
\right\}
\end{equation}

\item[$(XS)(Y)$] if and only if
\begin{equation} \label{conds:(XqS)(Y)}
\left.
\begin{array}{ll}
& (B_1): \quad I_3q^2-qI_{1y}+q_yI_1-2I_2 = 0 \\
& -1/2qq_yI_{1x}-3/2q_xqI_{1y}+I_5q^2-q^3I_{3x}+1/2q^2I_{1xy}+
q_xq^2I_3\\
& -2q_xI_2 +qI_{2x} +3I_1q_xq_y-3/2I_1q_{xy}q-I_2I_1+
2q_y^2q_xq-2q_yqI_2 \\
& -2q_{xy}q^2q_y+2q_{xy}qq_x-2q_yq_x^2 \ .
\end{array}
\right\}
\end{equation}

\item[$(Y)(SX)$] if and only if
\begin{equation} \label{conds:(Y)(XqS)}
\left.
\begin{array}{ll}
& (A_1): \quad I_3+q_{yy}= 0 \ , \\
& \quad
-qI_{4y}-3/2q_xqI_{1y}+I_5q^2+1/2q^2I_{1xy}+2q_yI_4-1/2qq_yI_{1x}+
\\
& 3I_1q_xq_y-3/2I_1q_{xy}q-q^2I_{2y} = 0 \ ;
\end{array}
\right\}
\end{equation}

\item[$(Y)(X)(S)$] if and only if
$ \quad (\ref{conds:(XY)(qS)}) \quad \& \quad
-qq_{xx}+I_4+2q_x^2+qI_2-qq_xq_y+q^2q_{xy} = 0\ ; $

\item[$(Y)(S)(X)$] if and only if
$ \quad (\ref{conds:(YqS)(X)}) \quad \& \quad (A_1) \ ; $

\end{description}
\end{theorem}

For LPDOs of the forms~(\ref{op:XXY}) and~(\ref{op:XXX}) generating
sets of invariants and the corresponding conditions of the existence
of factorizations of different types are know
also~\cite{movingframes}, \cite{inv_cond_non_hyper_case}.

\section{Several Factorizations of One Operator}
\label{sec:main}
\begin{theorem} \label{thm:main_3} Let $\gcd(S_1,S_2)=1$.
A third-order bivariate operator $L$ has a first-order left factor
of the symbol $S_1$ and a first-order right factor of the symbol
$S_2$ if and only if it has a complete factorization of the type
$(S_1)(T)(S_2)$, where $T=\Sym(L)/(S_1 S_2)$.
\end{theorem}
The following diagram is an informal illustration of the statement
of the theorem:
\[
\left( (S_1)(\dots) \quad \wedge \quad (\dots)(S_2)  \right) \iff
(S_1)(\dots)(S_2)
\]

\begin{proof} The part of the statement ``$\Longleftarrow$''
is trivial. Prove ``$\Longrightarrow$''.

The symbol of the operator has two different factors, therefore, the
normalized form of the operator $L$ is either~(\ref{L1})
or~(\ref{op:XXY}). Without loss of generality we can consider $L$ in
its normalized form.

Consider the first (hyperbolic) case. For this class of operators we
have the generating system of invariants $q, I_1, I_2, I_3, I_4,
I_5$ from Theorem~\ref{invariant_general}. The symbol is the same
for all the operators in the class and is $X \cdot Y \cdot S$, where
$S=pX+qY$. The following six cases are the only possibilities for
the $S_1$ and $S_2$, $\gcd(S_1,S_2)=1$.

Case $S_1=S$, $S_2=Y$. By the Theorem~\ref{thm:non-constant}
operator $L$ has a left factor of the symbol $S_1=S$ and a right
factor of the symbol $S_2=Y$ if and only if conditions
(\ref{conds:(X+qY)(XY)}) and (\ref{conds:(XqS)(Y)}) are satisfied.
From the second equality in~(\ref{conds:(XqS)(Y)}) derive an
expression for $I_3$ and substitute it for $I_3$ into the first
equality in~(\ref{conds:(X+qY)(XY)}). The resulting equality implies
that the third condition (in Theorem~\ref{thm:non-constant}) for the
existence of factorization of the type $(S)(X)(Y)$ is satisfied. The
remaining two first conditions are exactly the same as two
conditions~(\ref{conds:(X+qY)(XY)}), and therefore the theorem is
proved for this case.

Case $S_1=Y$, $S_2=S$. Operator $L$ has a left factor of the symbol
$S_1=Y$ and a right factor of the symbol $S_2=S$ if and only if
conditions (\ref{conds:(Y)(XqS)}) and (\ref{conds:(XY)(qS)}) are
satisfied. From the first equality in~(\ref{conds:(Y)(XqS)}) derive
an expression for $I_3$ ($I_3=-q_{yy}$) and substitute it for $I_3$
into the first equality in~(\ref{conds:(XY)(qS)}). The resulting
equality implies that the third condition (in
Theorem~\ref{thm:non-constant}) for the existence of factorization
of the type $(Y)(X)(S)$ is satisfied. Since the first two conditions
are satisfied obviously, we proved the theorem for this case.

Cases $S_1=X$, $S_2=Y$ and $S_1=Y$, $S_2=X$ and $S_1=X$, $S_2=S$ and
$S_1=S$, $S_2=X$ are obvious consequences of
Theorem~\ref{thm:non-constant}.

Consider the case, where the symbol of $L$ has exactly two different
factors, that is $L$ is in the normalized form~(\ref{op:XXY}). There
are only two cases to consider: $S_1=X, S_2=Y$, and $S_1=Y, S_2=X$.

Straightforward computations shows that the equality $H_1 \circ H_2
= G_2 \circ G_1$, where $\Sym(H_1)=X$ and $\Sym(G_1)=Y$ implies that
for some $a=a(x,y)$ and $b=b(x,y)$ we have $H_2 = D_{xy} + a D_x + b
D_y + a_x + ab=(D_x + b) \circ (D_y +a)$, while $G_1=D_y+a$. Thus,
$H_2$ has a factorization of the type $(X)(Y)$.

Similar computations shows that the equality $H_1 \circ H_2 = G_2
\circ G_1$, where $\Sym(H_1)=Y$ and $\Sym(G_1)=X$ implies that for
some $a=a(x,y)$ and $b=b(x,y)$ we have $G_2 = D_{xy} + a D_x + b D_y
+ b_y + ab=(D_y + a) \circ (D_x +b)$, while $H_1=D_y+a$. Thus, $H_2$
has a factorization of the type $(Y)(X)$.
\end{proof}

\begin{example}[Symbol $SXY$, $S_1=X+Y$, $S_2=Y$] We found an operator with two factorizations
$L=(D_x + D_y + x) \circ (D_{xy}+ yD_x +y^2 D_y + y^3)$ and
$L=(D_{xx}+ D_{xy} +(x+y^2)D_x + y^2 D_y + xy^2+2y) \circ (D_y +
y)$. Then $L$ has factorization $L=(D_x +D_y + x) \circ (D_x +y^2)
\circ (D_y +y)$.
\end{example}

\begin{example}[Symbol $X^2Y$, $S_1=X$, $S_2=Y$] We found an operator with two factorizations
$L=(D_x + x) \circ (D_{xy}+ yD_x +y^2 D_y + y^3)$ and $L=(D_{xx}+
(x+y^2)D_x + xy^2) \circ (D_y + y)$. Then $L$ has factorization
$L=(D_x + x) \circ (D_x +y^2) \circ (D_y +y)$.
\end{example}

\begin{example}[Symbol $X^2Y$, $S_1=Y$, $S_2=X$] We found an operator with two factorizations
$L=(D_y + x) \circ (D_{xx}+ yD_x +y^3-y^4)$ and $L=(D_{xy}+ xD_x
+y^2D_y +xy^2+2y) \circ (D_x + y-y^2)$. Then $L$ has factorization
$L=(D_y + x) \circ (D_x +y^2) \circ (D_x + y-y^2)$.
\end{example}

Looking at the examples, one can notice that the factorizations into
first-order factors have the right and left factors exactly the same
as they were in the initial, given factorizations. In fact, this
will be always the case. Accordingly, we improve
Theorem~\ref{thm:main_3} proving the following one.

\begin{theorem} \label{thm:improved1}
A third-order bivariate operator $L$ has a first-order left factor
$F_1$ and a first-order right factor $F_2$ with
$\gcd(\Sym(F_1),\Sym(F_2))=1$ if and only if $L$ has a factorization
into three factors, the left one of which is exactly $F_1$ and the
right one is exactly $F_2$.
\end{theorem}
The following diagram is an informal illustration of the statement
of the theorem:
\[
\left( L=F_1 \circ \dots \quad \wedge \quad L=\dots \circ F_2
\right) \iff L = F_1 \circ \dots \circ F_2 \ .
\]
\begin{proof} Let $L$ have the normalized
form~(\ref{L1}). Then by Theorem~\ref{thm:main_3} if $L$ has a
first-order left factor $F_1$ and a first-order right factor $F_2$
($\Sym(F_2)$ is co-prime with $\Sym(F_1)$), it has a factorization
into first-order factors of the type $(S_1)(R)(S_2)$, where
$R=\Sym_L/(S_1 S_2)$. Theorem~\ref{thm:gs} implies that such
factorization is unique, so we have some unique first-order LPDOs
$T_1, T, T_2$ such that $L= T_1 \circ T \circ T_2$, where
$\Sym(T_1)=S_1$, $\Sym(T)=R$, $\Sym(T_2)=S_2$. This also means that
there are factorization $L= T_1 \circ (T \circ T_2)$ of the type
$(S_1)(R S_2)$ and factorization $L= (T_1 \circ T) \circ T_2$ of the
type $(S_1R)(S_2)$. Since $S_1, R, S_2$ are pairwise coprime, by
Theorem~\ref{thm:gs} such factorizations are unique. On the other
hand we have initial factorizations that are factorizations of the
same types. Thus, we have $F_1=T_1$ and $F_2=T_2$.

For $L$ that has the normalized form~(\ref{op:XXY}), the statement
of the theorem is actually a subresult in the proof of
Theorem~\ref{thm:main_3} for this case.
\end{proof}

\begin{proposition}
The condition $\gcd(S_1,S_2)=1$ in Theorems~\ref{thm:main_3}
and~\ref{thm:improved1} cannot be omitted.
\end{proposition}
\begin{proof} Hyperbolic case. Consider an equivalence class of~(\ref{L1}) defined
by $q=1$, $I_1=I_2=I_5=0$, $I_3=I_4=x-y$ of the invariants from
Thereom~\ref{invariant_general}. Using
Theorem~\ref{thm:non-constant} one can verify that operators of the
class have factorizations of the types $(S)(XY)$ and $(XY)(S)$ only.

Such equivalence class is not empty. For example, operator $A_3 =
D_{xxy} + D_{xyy} + (x-y)(D_x+D_y)$ belongs to this equivalence
class. Only the following two factorizations exist for $A_3$: $A_3=
(D_{xy} + x-y) (D_x + D_y) = (D_x + D_y) (D_{xy} + x-y)$.

The non-hyperbolic case. Consider operator of Landau
\[
D_x^3+x D_x^2D_y+2 D_x^2+(2x+2)D_xD_y+D_x+(2+x)D_y \ ,
\]
which has two factorizations into different numbers of irreducible factors:
\[
L = Q \circ Q \circ P = R \circ Q \ ,
\]
for the operators $P = D_x +xD_y, \quad Q = D_x +1, \quad  R = D_{xx} +x D_{xy}
+D_x +(2+x)D_y$. That is factorizations of the types $ (X)(SX)$, $(SX)(X) $
exist, while those of the type $(X)(S)(X)$ do not. Here we denote $S=X+xY$.
\end{proof}

\begin{proposition}  The statement of Theorem~\ref{thm:improved1}
is not always true for a general
fourth-order hyperbolic operator.
\end{proposition}
\begin{proof} For example, operator
\begin{eqnarray*}
&& L =(D_x + D_y)\circ ( D_x D_y (D_x + D_y) + x D_{xx} + (2- x^2)
D_x + x D_y - 2x +
  x^2) \\
  && = (D_x (D_x + D_y)^2  - x D_{x}(D_x + D_y) + (x-2)D_x + (x-1)D_y +1) \circ (D_y +
  x) \ .
\end{eqnarray*}
The second factor in the first factorization has no factorization.
\end{proof}

\section{Completely Reducible Operators}
\label{sec:completely_reducible}

Let $<L>$ denote the left ideal generated by an operator $L \in
K[D]$. Consider Linear Ordinary Differential Operators (LODOs). The
ring of LODOs are the principal ideal domain and, therefore, the
intersection of two principal ideals is again principle.
Consequently, the least common multiple ($\lcm$) of two LODOs $L_1$
and $L_2$ can be defined uniquely as $L$ such that $<L> = <L_1> \cap
<L_2>$. Since in the ring of LPDOs this is not the case, it was
suggested\cite{gs} to introduce the notion of a completely
irreducible LPDO.

\begin{definition} \cite{gs} An LPDO $L$ is said to be completely irreducible, if  it can be
expressed as $<L> = <L_1> \cap \dots \cap <L_k>$
  for suitable irreducible LPDOs $L_1,
\dots , L_k$. In this case $ L = \lcm \{ L_1, \dots, L_k \} $ by definition.
\end{definition}

\begin{theorem} \cite{gs}  \label{thm:gs_inters_ideals}
If an LPDO $L$ has right factors $L_1, \dots, L_k$ and
\begin{equation} \label{cond:sym=lcm}
\Sym_L = \lcm(\Sym_{L_1}, \dots, \Sym_{L_k}) \ ,
\end{equation}
then $<L> = <L_1> \cap \dots \cap <L_k>$. If the factors $L_1, \dots, L_k$ are
irreducible, then $L$ is completely reducible via $L_1, \dots, L_k$.
\end{theorem}

An additional piece of motivation is~\cite{gs} the following. Let
for an ideal $I \subset K[D]$ denote by $V_I \subset K$ its space of
solutions. Then for two ideals $I_1, I_2 \subset K[D]$ we have
\cite{Cassidy:72,Sit:74} $ V_{I_1 \cap I_2} = V_{I_1} + V_{I_2} $,
which allows to reduce the solution problem of the partial
differential equation corresponding to a completely reducible LPDO
to ones of corresponding to its factors.

Notice that the properties of the existence of a right factor with
certain symbol or a factorization of certain factorization type,
and, therefore, irreducibility of factors are invariant under the
gauge transformations. Consequently, an invariant description of the
completely reducible operators is possible.


 Consider a hyperbolic linear partial differential operator of third order in
the normalized form~(\ref{L1}). Consider all possible right factors
of the operator, their symbols are
\[
X \ ,  Y \ ,  S \ ,   XY \ ,  XS \ , YS \ ,
\]
where $S=X+qY$. Let us list all the possibilities for $\{ L_1,
\dots, L_k \}$ that $L_i$-s together satisfy~(\ref{cond:sym=lcm}).
Notice that the number (of factors) $k$ is not fixed. However, by to
Thereom~\ref{thm:gs} there is no more than one factorization of each
factorization type. Thus, $L_i \neq L_j$  for $i \neq j$, and,
therefore, $k \leq 6$.
\begin{enumerate}
\item[I.]
$\{ X, Y, S \}$ \ ;
\item[II.]
 $\{ SX, SY \}$ \ , $\{ SX, XY \}$ \ , $\{ SY, XY \}$ \ ;
\item[III.]
 $\{ X, SY \}$ \ , $\{ Y, SX \}$ \ , $\{ S, XY \}$ -- the right factors are co-prime \ ;
\item[IV.] $\{ X, Y, SX \}$ \ , $\{ X, XY ,SY \} \ , \dots$  -- the sets that
contain as the subset one of the sets of the groups $III$ and $I$.
\item[V.] $\{SX, SY, XY \}$ -- the only set that
contain as the subset one of the sets of the groups $II$ and does
not belong to the group $IV$.
\end{enumerate}

Theorem~\ref{thm:main_3} allows us to avoid the consideration of the
large group $IV$. Indeed, By Thereom~\ref{thm:main_3}, at least one
of the second-order factors of the sets fail to be irreducible.

Now, when we rule out all cases but eight, using
Theorem~\ref{thm:non-constant} it is easy to obtain sufficient
conditions for LPDOs to be completely reducible with
\[
 <L> = <L_1> \cap  \dots \cap <L_k> \ ,
\]
where $\{ \Sym(L_1), \dots ,  \Sym(L_k) \}$ belong to $I$ or $II$ or
$III$  or $V$. We just combine certain conditions from
Theorem~\ref{thm:non-constant}. Further below we use notation
$\overline{W}$ for an arbitrary operator with the principal symbol
$W$.

It is of interest to consider instead an important particular case
$q=1$. We collect the sufficient conditions for this case in the
following theorem.

\begin{theorem} Given an equivalence class of (\ref{L1}) by
$q=1$, and values $I_1, \dots , I_5$ of the invariants from
Theorem~\ref{invariant_general}. Operators of the class are
completely reducible with

$I$.  $<L> = <\overline{X}> \cap  <\overline{Y}> \cap
<\overline{S}>$ if
\begin{eqnarray*}
 && (D_y +I_1) \circ (2 D_x + D_y) (I_1) = 0 \ , \\
 && I_2 = I_{1x}-I_3 \ , \\
&& I_3 = (I_{1y}+2I_{1x})/3 \ , \\
&& I_4 = -I_2+I_3 \ , \\
&& I_5 = I_1I_{1x}-2I_1I_3-I_{1xy}/2 \ ;
\end{eqnarray*}

$II$.  $<L> = <\overline{SX}> \cap  <\overline{SY}>$ if
\begin{eqnarray*}
 && I_2 = F_1(y-x) \ , \\
&& I_3 = I_4 = 0\ , \\
&& I_5 = -1/2I_{1xy}+I_{2y}\ ,
\end{eqnarray*}
where $F_1(y-x)$ is some function;

 $<L> = <\overline{SX}> \cap  <\overline{XY}>$ if
\begin{eqnarray*}
 && I_{1xy}+I_{2x}-I_1I_{1y}-2I_1I_2 = 0 \ , \\
 && I_3 = 0 \ , \\
 && I_4 =  -I_{1y}+I_{1x}-3I_2\ , \\
 && I_5 = I_{4y}-I_{1xy}/2+I_{2y} \ ;
\end{eqnarray*}

$<L> = <\overline{SY}> \cap  <\overline{XY}>$ if
\begin{eqnarray*}
 && I_{1xy}-I_1I_{1x}-I_{2y}+I_1I_2= 0 \ , \\
 && I_3 = I_{1y}-I_{1x}+3I_2 \ , \\
 && I_4 = 0 \ , \\
 && I_5 =  I_{3x}-1/2I_{1xy}-I_{2x} \ ;
\end{eqnarray*}

$III$.  $<L> = <\overline{X}> \cap  <\overline{SY}>$ if
\begin{eqnarray*}
 && -I_{3x}+(I_{1x}I_1+I_{1xy}+I_{1xx})/2 = 0\ , \\
 && I_2 = I_{1x}/2 \ , \\
&& I_4 = 0\ , \\
&& I_5 = -1/2I_{1xy}+I_2I_1+I_{2y} \ ;
\end{eqnarray*}

$<L> = <\overline{Y}> \cap  <\overline{SX}>$ if
\begin{eqnarray*}
 && -I_{4y}-I_{1y}I_1/2+I_{1xy}/2+I_{1yy}/2= 0\ , \\
 && I_2 = -I_{1y}/2 \ , \\
&& I_3 = 0\ , \\
&& I_5 = I_2I_1-I_{1xy}/2-I_{2x} \ ;
\end{eqnarray*}

$<L> = <\overline{S}> \cap  <\overline{XY}>$ if
\begin{eqnarray*}
 && -I_{3y}+(I_{1xy}+I_{1xx}-I_1I_{1y}-I_1I_{1x})/2-I_{3x}= 0\ , \\
 && I_2 = (I_{1x} -I_{1y})/2 \ , \\
&& I_4 = -I_2+I_3\ , \\
&& I_5 = ( I_1I_{1x}-I_1I_{1y}-I_{1xy})/2-I_1I_3 \ ;
\end{eqnarray*}

$V$. $<L> = <\overline{SX}> \cap  <\overline{SY}> \cap
<\overline{XY}>$ if
\begin{eqnarray*}
 && I_{1xx}-I_{1yy}= 0 \ , \\
 && I_{1xx}-2I_1I_{1x}+2I_{1xy}-I_1I_{1y} = 0 \
 , \\
 &&  I_2 = (I_{1x}-I_{1y})/3 \ ,  \\
 &&  I_3=I_4 = 0 \ , \\
 &&  I_5 = -I_{1xy}/2+I_{2y} \ .
\end{eqnarray*}
\end{theorem}
\section{Conclusions}

The paper is devoted to the case when one LPDO has several
factorizations.

In Sec.~\ref{sec:main} we proved that a third-order bivariate
operator $L$ has a first-order left factor $F_1$ and a first-order
right factor $F_2$ with $\gcd(\Sym(F_1),\Sym(F_2))=1$ if and only if
$L$ has a factorization into three factors, the left one of which is
exactly $F_1$ and the right one is exactly $F_2$. Also it was shown
that the condition $\gcd(\Sym(F_1),\Sym(F_2))=1$ is essential, and
that the analogous statement ``as it is'' is not true for LPDOs of
order four. However, other generalizations may be possible.

The proof for the hyperbolic case was done using invariants'
methods. This is a nice and easy way to prove the things since the
expressions for the necessary and sufficient conditions of the
existence of factorizations of a given type are already known. It
was the form of the conditions, that allowed us to make initial
hypothesis that were proved later to be true in Sec.~\ref{sec:main}.
However, some other method is required for generalizations to higher
order LPDOs.

In Sec.~\ref{sec:completely_reducible} we considered the case, where
one LPDO has two or more several factorizations of certain types.
Most of the cases were ruled out from the consideration due to the
results of Sec.~\ref{sec:main}. The explicit formulae for the
sufficient conditions for the complete reducibility of an LPDO were
found for the case $p=1$ (which is the case where the symbol of $L$
has constant coefficients only).

\subsubsection*{Acknowledgments.}
The author was supported by the Austrian Science Fund (FWF) under
project DIFFOP, Nr. P20336-N18.


\end{document}